\documentclass[10 pt]{amsart}
\usepackage{amsmath}
\usepackage{amssymb}
\usepackage{amsthm}
\usepackage{array}
\usepackage{xy}
\newcommand{\CC}{\mathbb{C}}
\newcommand{\NN}{\mathbb{N}}
\newcommand{\ZZ}{\mathbb{Z}}
\newcommand{\ot}{\otimes}

\begin{document}
% nohead,nofoot

\author{Noah Arbesfeld, David Jordan}
\date{}
\title[Lower central series quotients of a free associative algebra]{New results on the lower central series quotients of a free associative algebra}
\maketitle

\begin{abstract}
We continue the study of the lower central series and its associated graded components for a free associative algebra with $n$ generators, as initiated in \cite{FS}.  We establish a linear bound on the degree of tensor field modules appearing in the Jordan-H\"older series of each graded component, which is conjecturally tight.  We also bound the leading coefficient of the Hilbert polynomial of each graded component.  As applications, we confirm conjectures of P. Etingof and B. Shoikhet concerning the structure of the third graded component.
\end{abstract}

\newtheorem{thm}{Theorem}[section]
\newtheorem{defn}[thm]{Definition}
\newtheorem{lem}[thm]{Lemma}
\newtheorem{prop}[thm]{Proposition}
\newtheorem{cor}[thm]{Corollary}
\newtheorem{conj}[thm]{Conjecture}

\theoremstyle{remark}
\newtheorem{rem}[thm]{Remark}
\newtheorem{ex}[thm]{Example}

\section{Introduction and results}

In this paper we consider the free associative algebra\footnote{over $\CC$, or any field of characteristic zero} $A:=A_n$  on generators $x_1,\ldots, x_n$, for $n\geq 2$, and its lower central series filtration: $L_1=A, L_{m+1}=[A,L_m]$.  The corresponding associated graded Lie algebra is $B(A)=\oplus_m B_m(A)$, where $B_m(A) = L_m(A)/L_{m+1}(A)$.  The natural grading on $A$ by $S=(\ZZ_{\geq 0})^n$ descends to each $B_m$, and it is interesting to study the Hilbert series $h_{B_m}(t_1,\ldots,t_n)$ and $h_{B_m}(t):=h_{B_m}(t,\ldots,t)$.
%$$h_M(x_1,\ldots, x_n) = \sum_{I\in S} \dim M[I] x^I,$$ where $x^{(i_1,\ldots,i_n)}=x_1^{i_1}\cdots x_n^{i_n}$, as well as their singly graded Hilbert series:
%$$h_M(t) = \sum_{d\in \ZZ_{\geq 0}} \dim M[d] t^d = h_M(t,\ldots,t).$$
The formula for $h_A$ is straightforward: 
$$h_A(t_1\ldots, t_n)=\frac{1}{1-(t_1 + \cdots + t_n)},$$ and implies that $\dim A[d]$ grows exponentially in $d$. It is thus a somewhat surprising fact that for $m\geq 2$, $\dim B_m[d]$ grows as a polynomial in $d$ of degree $n-1$.  For $m=2$, this was shown in \cite{FS}, and for $m\geq 3$ it was conjectured in \cite{FS} and proven in \cite{DE}.

The proof of this fact is based on the representation theory of the Lie algebra $W_n$ of polynomial vector fields on $\mathbb{C}^n$.  Namely, in \cite{FS}, an action of $W_n$  was constructed on each $B_m, m\geq 2$.  It was conjectured there, and proved in \cite{DE}, that each $B_m$ had a finite length Jordan-H\"older series, with respect to this action.  The proof relied, firstly, on the observation in \cite{FS} that all irreducible subquotients of $B_m$ could be identified with certain tensor field modules $\mathcal{F}_\lambda$ associated to a Young diagram $\lambda$, and secondly, on a bound for the sizes $|\lambda|$ that could occur:

\begin{thm}\cite{DE} For $m\geq 3$ and $\mathcal{F}_\lambda$ in the Jordan-H\"older series of $B_m(A_n)$, we have the following estimate on the size (i.e., the number of squares) of the Young diagram $\lambda$:
$$ |\lambda| \leq (m-1)^2 + 2\lfloor\frac{n-2}{2}\rfloor (m-1)$$
(where $\lfloor x\rfloor$ denotes the integer part of $x$).
\end{thm}
This result, combined with well-known formulas for the Hilbert series of each $\mathcal{F}_\lambda$, established the finiteness of the Jordan-H\"older series, as well as the growth of $\dim B_m[d]$ as a degree $n-1$ polynomial in $d$, for $m\geq 2$. However, it was evident from experimental computations produced by Eric Rains that the maximal $|\lambda|$ for which $\mathcal{F}_\lambda$ occurs in $B_m$ should grow linearly rather than quadratically in $m$.  This indeed turns out to be the case.  Namely, the main result of this paper is the following improvement of the bound of \cite{DE}:

\begin{thm}\label{2m-3} Let $m\geq 3$.
\begin{enumerate}
\item For $\mathcal{F}_\lambda$ in the Jordan-H\"{o}lder series of $B_m(A_n)$, $$|\lambda| \leq 4m-7 + 2\lfloor \frac{n-2}{2}\rfloor.$$
\item Let $n=2$ or $3$.  For $\mathcal{F}_{\lambda}$ in the Jordan-H\"{o}lder series of $B_m(A_n)$, $$|\lambda| \leq 2m-3$$
\end{enumerate}
\end{thm}

This is proven by means of some elementary commutative algebra, and a technical but very useful result:

\begin{thm}\label{xyxy} Let $m\geq2$.
\begin{enumerate}
\item For all $n$, we have:
$$B_2=\sum_i[x_i,B_1].$$
\item For $n\geq 4$, we have:
$$B_{m+1}= \sum_{i}[x_i,B_m]+\sum_{i\leq j}[x_ix_j,B_m]+\sum_{i<j<k}[x_i[x_j,x_k],B_m].$$
\item For $n=3$, we have:
$$B_{m+1}=\sum_i[x_i,B_m]+\sum_{i \leq j}[x_ix_j,B_m].$$ 
\item For $n=2$, we have:
$$B_{m+1}= \sum_i[x_i,B_m]+\sum_{i < j}[x_ix_j,B_m].$$ 
\end{enumerate}
\end{thm}

\begin{rem} The presence of the cubic terms in (2) above appears superfluous from computer experiments, and we conjecture that they can be omitted (see also Lemma \ref{B3identity} and Corollary \ref{B3case}).  Were this so, it would imply the bound $|\lambda|\leq 2m-3 + 2\lfloor\frac{n-2}{2}\rfloor$ in Theorem \ref{2m-3}.
\end{rem}

\begin{defn} The Hilbert polynomials $p_{mn}(d)$ are defined by
$$p_{mn}(d)=\dim B_m(A_n)[d], \quad (d\gg 0).$$
The \emph{density}, $a_{mn}$, is the leading coefficient of $p_{mn}$, times $(n-1)!$.
\end{defn}

\begin{ex} The Hilbert polynomial of $\CC[x_1,\ldots,x_n]$ is ${n+d-1}\choose {n-1}$, with leading coefficient $\frac{1}{(n-1)!}$, so the density is one.  More generally, if $\lambda_1\geq 2$ or $\lambda=(1^n)$, the density of $\mathcal{F}_\lambda$ is equal to the dimension of the irreducible representation $V_\lambda$ of $\mathfrak{gl}_n$ with highest weight $\lambda$.
\end{ex}

As a corollary to Theorem \ref{xyxy}, we derive a bound for the density $a_{mn}$:

\begin{cor}\label{rank} 
For $n=2$, we have $$a_{m+1,n} \leq 3 a_{m,n}.$$
For $n=3$, we have $$a_{m+1,n} \leq 9 a_{m,n}.$$
For $n\geq 4$, we have $$a_{m+1,n} \leq \frac{n^3+11n}{6}a_{m,n}.$$
\end{cor}

As applications of the theory, we are able to prove the following conjecture of P. Etingof describing the complete structure of $B_3(A_n)$:

\begin{thm}\label{PavelConj}
$$B_3(A_n)=\bigoplus_{i=1}^{\lfloor\frac{n}{2}\rfloor}(2,1^{2i-1},0^{n-2i}).$$
\end{thm}

As a corollary we derive the following conjecture of B. Shoikhet, which motivated the first conjecture:

\begin{cor}\label{BConj}  Let $B_3(A_n)[1,\ldots,1]$ denote the subspace of $B_3(A_n)$ of degree 1 in each generator.  We have
$$\dim B_3(A_n)[1,\ldots,1]=(n-2)2^{n-2}.$$
\end{cor}

\noindent Here, as elsewhere in the paper, we have used the abbreviation $(p_1,\ldots,p_n)$ instead of $\mathcal{F}_{(p_1,\ldots,p_n)}$. Combining our bounds in Theorem \ref{2m-3} with MAGMA\cite{BCP} computations, we are also able to give the complete Jordan-H\"older series of $B_m(A_n)$ for many new $m$ and $n$. 

The structure of the paper is as follows.  In Section \ref{pre} we briefly review the representation theory
of the Lie algebra $W_n$, as well as the results of \cite{FS} we will use. In Section \ref{xyxypf} we prove Theorem \ref{xyxy} and Corollary \ref{rank}.  In Section \ref{2m-3pf} we prove Theorem \ref{2m-3}.  In Section \ref{PavelConjPf}, we prove Theorem \ref{PavelConj}, and present as a corollary a geometric description of the bracket map of $\bar{B}_1$ with $B_2$.  In Section \ref{decomps}, we present the Jordan-H\"older series for $B_m(A_n)$ for small $m$ and $n$.

\subsection{Acknowledgments}  We would like to heartily thank Pavel Etingof for many helpful conversations as the work progressed, and especially for explaining how to derive Theorem \ref{PavelConj} from our work.  The work of both authors was funded by the Research Science Institute, at MIT.  The research of the second author was partially supported by the NSF grant DMS-0504847.

\section{Preliminaries}\label{pre}
In this section, we recall definitions for the Lie algebra $W_n$, the tensor field modules $\mathcal{F}_\lambda$, and the quantized algebra of even differential forms $\Omega_*^{ev}$.

\begin{defn} Let $W_n=Der(\CC[x_1,\ldots,x_n])$ denote the Lie algebra of polynomial vector fields,
$$W_n = \oplus_i\CC[x_1,\ldots,x_n] \partial_i,$$
with bracket $[p\partial_i,q\partial_j]=p\frac{\partial q}{\partial x_i} \partial_j - q\frac{\partial p}{\partial x_j}\partial_i$.
\end{defn}

Let $\mathfrak{gl}_n$ denote the Lie algebra of $n$ by $n$ matrices. A Young diagram
$$\lambda=(\lambda_1\geq \lambda_2\geq \cdots \geq \lambda_n)$$
with $n$ rows gives rise to an finite dimensional irreducible representation $V_{\lambda}$ of $\mathfrak{gl}_n$ contained in the space $(\mathbb{C}^{n*})^{\otimes |\lambda|}$ of covariant tensors of rank $|\lambda|$ on $\mathbb{C}^n$. Let $\widetilde{\mathcal{F}}_{\lambda}$ be the space of polynomial tensor fields of type $V_{\lambda}$ on $\mathbb{C}^n$. As a vector space, 
$$\widetilde{\mathcal{F}}_{\lambda}=\mathbb{C}[x_1,\ldots,x_n]\otimes V_{\lambda}.$$ 

It is well known that $\widetilde{\mathcal{F}}_{\lambda}$ is a representation of $W_n$ with action given by the standard Lie derivative formula for action of vector fields on covariant tensor fields (see, e.g. \cite{SL} for details).  

\begin{thm}\cite{ANR}
If $\lambda_1 \geq 2$, or if $\lambda=(1^n)$, then $\widetilde{\mathcal{F}}_{\lambda}$ is irreducible. Otherwise, if $\lambda =(1^k,0^{n-k})$, then $\widetilde{\mathcal{F}}_{\lambda}$ is the space $\Omega^k=\Omega^k(\mathbb{C}^n)$ of polynomial differential k-forms on $\mathbb{C}^{n}$, and it contains a unique irreducible submodule which is the space of all closed differential k-forms.
\end{thm}

Denote by $\mathcal{F}_{\lambda}$ the irreducible submodule of $\widetilde{\mathcal{F}}_{\lambda}$, so that $\mathcal{F}_{\lambda}=\widetilde{\mathcal{F}}_{\lambda}$ unless \mbox{$\lambda = (1^k,0^{n-k})$} for some $1\leq k \leq n-1$.

\begin{thm}\label{semisimple}\cite{ANR}
Any $W_n$-module on which the operators $x_i\partial_i$, for $i=1,\ldots,n$, act semisimply with 
nonnegative integer eigenvalues and finite dimensional common 
eigenspaces has a Jordan-H\"{o}lder series
whose composition factors are $\mathcal{F}_{\lambda}$, each occurring with finite
multiplicity. \end{thm}

As a trivial, but important, consequence of the definition of $\mathcal{F}_\lambda$, we have:

\begin{prop}
Suppose $\lambda\neq (1^k,0^{n-k})$ for $1\leq k \leq n-1$.  Let $h_{V_\lambda}$ denote the Hilbert series for $V_\lambda$.  Then we have
$$h_{\mathcal{F}_\lambda}(t_1,\ldots, t_n) = \frac{h_{V_\lambda}(t_1,\ldots,t_n)}{(1-t_1)\cdots(1-t_n)},$$
\end{prop}

\noindent where we view $V_\lambda$ as a $\mathbb{Z}^n$-graded vector space via its weight decomposition. In particular, we observe that $h_{\mathcal{F}_\lambda}(t_1,\ldots,t_n)(1-t_1)\cdots(1-t_n)$ is a polynomial of total degree $|\lambda|$.

Let $\Omega^{ev}$ be the space of polynomial differential forms on $\mathbb{C}^n$ of even rank, and $\Omega^{ev}_{ex}$ its subspace of even exact forms. These spaces are graded by setting deg$(x_i)=$deg$(dx_i)=~1$. Recall that $\Omega^{ev}$ is a \emph{commutative} algebra with respect to the wedge product.
\begin{defn}
The multiplication $a*b=a\wedge b+da\wedge db$ defines an associative product on $\Omega^{ev}$. By $\Omega^{ev}_*$ we will mean the space $\Omega^{ev}$ equipped with the $*$-product, and call it the quantized algebra of even differential forms. 
\end{defn}

\begin{defn}
Let $Z$ denote the image of $A[A,[A,A]]$ in $B_1$, which was shown in \cite{FS} to be central in $B$.  We define $\bar{B}_{1}$ to be the quotient, $\bar{B}_{1}=B_{1}/Z,$ and define $\bar{B}=\bar{B}_1~\oplus~(\displaystyle\oplus_{i \geq 2} B_i)$. \end{defn}
Clearly $\bar{B}$ inherits the grading from $B$, and is thus a graded Lie algebra generated in degree 1.

\begin{thm}\label{isom}\cite{FS}
There is a unique isomorphism of algebras,
\begin{align*}
\xi: \Omega^{ev}_* &\to A/A[A,[A,A]],\\
x_i &\mapsto x_i,
\end{align*}
which restricts to an isomorphism $\xi:\Omega_{*,ex}^{ev}\overset\sim\to B_2$, and descends to an isomorphism $\xi: \Omega_*^{ev}/\Omega^{ev}_{*,ex}\overset\sim\to\bar{B}_1$.
\end{thm}

\begin{thm}\cite{FS}
 The action of $W_n$ on $\bar{B}_1 \cong \Omega^{ev}/\Omega^{ev}_{ex}$ by Lie derivatives uniquely extends to an action of $W_n$ on $\bar{B}$ by grading-preserving derivations.\end{thm}
 Thus, for each $m\geq2$, $B_{m}$ is a $W_n$-module clearly satisfying the conditions of Theorem \ref{semisimple}, and thus has composition factors $\mathcal{F}_{\lambda}$, each occurring with finite multiplicity. 

\section{Proof of Theorem \ref{xyxy} and Corollary \ref{rank}}\label{xyxypf}
 
\begin{lem}\label{identity} We have the following identity, which may be directly checked.
\begin{align*}{[u^{3},[v,w]]}=&{3[u^{2},[uv,w]]}-{3[u,[u^{2}v,w]]}
+\frac{3}{2}[u^2,[v,[u,w]]]\\&- \frac{3}{2}[u,[v,[u^{2},w]]]+[u,[u,[u,[v,w]]]]\\&
-\frac{3}{2}[u,[u,[v,[u,w]]]]+\frac{3}{2}[u,[v,[u,[u,w]]]].\end{align*}
\end{lem}
\begin{cor}\label{symm} Let $S(a,b,c)$ be the symmetrized sum  $$S(a,b,c)=\frac16(abc+bca+cab+acb+cba+bac).$$ Then, for all $a,b,c$, we have that $$[S(a,b,c),B_m] \subset [ab,B_m]+[bc,B_m]+[ca,B_m]+[a,B_m]+[b,B_m]+[c,B_m]\subset B_{m+1}.$$
\end{cor}
\begin{proof}
 In Lemma \ref{identity}, set $u=t_1a+t_2b+t_3c$, $v$ equal to any element of $A_n$, $w$ equal to any element of $B_{m-1}$, and take the coefficient of $t_1t_2t_3$. The result follows.
\end{proof}
\begin{lem}\label{lemmae}
 Let $E\subset \Omega^{ev}_*$ be the span of $S(a,b,c)$, where $a,b,$ and $c$ are of positive degree \emph{(}$\deg(x_i)=\deg(dx_i)=1$\emph{)}.   Let $X$ be the span of 1, $x_i$, $x_ix_j$ for $i\leq j$, $x_i dx_j \wedge dx_k$ for $i<j<k$. Then $\Omega^{ev}_*=X+E+\Omega^{ev}_{ex,*}$.
\end{lem}
\begin{proof} $\Omega^{ev}_*$ has a finite length descending filtration by rank of forms, such that the associated graded is the usual commutative algebra of even forms. Therefore, it is sufficient to check the same statement for the commutative algebra of even forms. Then $\Omega^{ev}/E$ is spanned by $$1, ~x_i, ~dx_i \wedge dx_j, ~x_ix_j, ~x_idx_j\wedge dx_k, ~dx_i\wedge dx_j \wedge dx_k \wedge dx_l.$$ As the forms $$dx_i\wedge dx_j, ~{x_idx_i\wedge dx_j}, ~dx_i\wedge 
dx_j\wedge dx_k\wedge dx_l$$ are exact, the Lemma follows. 
\end{proof}

Now we proceed to prove the theorem.  (1) Follows from the isomorphism $B_2\cong\Omega^{ev}_{ex}$.  For (2), we have:  $B_{m+1}=[\xi(\Omega^{ev}),B_m]$. Now, by Lemma \ref{lemmae} and Corollary \ref{symm}, $[\xi(\Omega^{ev}),B_m]=[\xi(X),B_m]$, so the statement follows.

For (4), we observe the following identity, which may be checked directly:
\begin{lem}\label{second}\begin{align*}
{[x^{2},[y,w]]}=&{2[x,[xy,w]]}+{[y,[x^2,w]]}-{2[xy,[x,w]]}-[w,[x,[y,x]]].\\
6[x^2,[xy,w]]=&12[x,[x^2y,w]] + 4[y,[x^3,w]] -6[xy,[x^2,w]]\\
& -3[x^2,[y,[x,w]]] - 3[y,[x^2,[x,w]]] + 9[x,[y,[x^2,w]]]\\
& -3[x,[x,[x,[y,w]]]] +3[x,[x,[y,[x,w]]]] -3[x,[y,[x,[x,w]]]]\\&-[y,[x,[x,[x,w]]]].
\end{align*}
\end{lem}
Setting $x=x_1$, $y=x_2$, and letting $w\in B_{m-2}$, we get that $[x_i^2,B_{m-1}]\subset [x_1,B_{m-1}]+[x_2,B_{m-1}]+[x_1x_2,B_{m-1}]$ as desired.

For (3), we use a lemma:

\begin{lem}\cite{DE} There exists a function $\epsilon:S_{m}\to\mathbb{Q}$, such that
$$[a_0,[a_{1},[\cdots [a_{m-1},a_m]\cdots] = \sum_{\sigma \in S_{m}}\epsilon(\sigma)[a_{\sigma(1)},[a_{\sigma(2)},[\cdots[a_{\sigma(m)},a_0]\cdots].$$
\end{lem}

Thus an element $x=[a_0,[a_{1},[\cdots [a_{m-1},a_m]\cdots]$, with $a_0\in\Omega^2$ is a sum of elements with $a_0$ in the innermost bracket.  The innermost bracket will then be in $\zeta(\Omega^4)$, which is zero in $B_2$ for $n=3$.
\qed

Corollary \ref{rank} follows immediately by counting the dimension of $X$.
%:  for $n=2$, we have $\frac{n(n-1)}{2}$ expressions $x_ix_j$, plus $n$ expressions $x_i$, and for $n=3$ we have $n$ additional terms $x_i^2$.  For $n\geq 4$ we have $n$ expressions $x_i$, $\frac{n(n+1)}{2}$ expressions $x_ix_j$, and $\frac{n(n-1)(n-2)}{6}$ expressions $x_idx_j\wedge dx_k$.

\section{Proof of Theorem \ref{2m-3}}\label{2m-3pf}
Let $\xi: \Omega^{ev} \to \bar{B}_1$ be the Feigin-Shoikhet surjective map from Theorem \ref{isom}.
Consider the map $f_m: (\Omega^{ev})^{\otimes m}\to B_m$
given by

$$f_m(a_1,\ldots,a_m)=[\xi(a_1),[\xi(a_2),\ldots[\xi(a_{m-1}),\xi(a_m)]].$$

\noindent Since $\bar{B}$ is generated in degree 1, this map is surjective.  By Theorem \ref{xyxy}, the restriction of $f_m$ to $Y:=(\Omega^0)^{\otimes m}$
is surjective for $n=2,3$, while the restriction to $Y:=(\Omega^0+\Omega^2)^{m-2}\ot(\oplus_{j+k\leq \lfloor \frac{n-2}{2}\rfloor}\Omega^{2j}\ot\Omega^{2k})$ is surjective for $n\geq 4$.  The idea of the proof will be to find a large $W_n$-submodule $K$ of $Y$ such that $f_m|_K=0$ and, such that the composition factors $\mathcal{F}_{\lambda}$ of  $Y/K$ satisfy the bound on $|\lambda|$ given in the theorem. This will obviously imply Theorem \ref{2m-3}.
\begin{lem}\label{cor1}
 Let $K \subset Y$ be the submodule spanned by elements of the form: 
\begin{enumerate} \item{$p_1\otimes\cdots\otimes p_{m-3-i}\otimes(1\otimes b)*(a\otimes1-1\otimes a)^3\otimes p_{m-2-i}\otimes\cdots\otimes p_{m-2},$ for $0\leq i\leq m-3,$ where $a\in\Omega_*^{0}, b\in \Omega_*^{ev},$ and $p_i\in \Omega_*^{ev}$, \emph{and}}
\item{$p_1\otimes\cdots\otimes p_{m-2}\otimes(b_1\otimes b_2)*(a\otimes1-1\otimes a)^2,$ where $a\in\Omega_*^{0}, b_1,b_2\in \Omega_*^{ev},$ and $p_i\in \Omega_*^{ev}$.}
\end{enumerate}
Then $f_m|_K=0$.
\end{lem}
\begin{proof}
 Elements of type 1 are killed by $f_m$ by Lemma \ref{identity}, as $$(1\otimes b)*(a\otimes  1-1\otimes a)^3=(a^3\otimes b-3a^2\otimes b*a+3a\otimes b*a^2-1\otimes b*a^3).$$ Elements of type 2 are killed, as, for any functions $a,b_1,b_2$, $$d(b_1a^2)\wedge db_2-2d(b_1a)\wedge d(b_2a)+db_1\wedge d(b_2a^2)=0. $$\end{proof} 

Let $K_0\subset Y$ be the associated graded of $K$ under the ``rank of forms'' filtration, and let $K_0'$ be the span of elements of type $1$ and $2$ in the commutative algebra of forms. The Jordan-H\"{o}lder series of $Y/K$ is the same as the Jordan-H\"{o}lder series of $Y/K_0$, so it is dominated by the Jordan-H\"{o}lder series of $Y/K_0'$. So, it suffices to show that all $\mathcal{F}_{\lambda}$ occurring in $Y/K_0'$ satisfy the bound of the theorem. We will do this by precisely computing the Hilbert series of $Y/K_0'$.  First, we recall a well-known fact from commutative algebra.

\begin{lem}\label{cor2} Let $B$ be a commutative algebra, and $I\subset B\otimes B$ be the kernel of the multiplication homomorphism 
$\mu:B\otimes B \to B.$ Also, let $k\in \mathbb{N}$. Then, $I^k$ is spanned by elements of the form $(1\otimes b)(a\otimes 1-1\otimes a)^k$, where $a,b \in B$.   
\end{lem}
\begin{proof} Obviously $(a\otimes1-1\otimes a)\in I$, so $(1\otimes b)(a\otimes 1-1\otimes a)^k\in I^k$. Now $I^k$ is generated by $(a_1\otimes 1-1\otimes a_1)\cdots(a_k\otimes 1-1\otimes a_k)$, for $a_i\in B$. Because for every vector space $V$, $S^k(V)$ is spanned by \{$v^{\otimes k}|v\in V$\}, such elements can be obtained as linear combinations of elements of the form $(a\otimes1-1\otimes a)^k$. So, $I^k$ is spanned by $(b_1 \otimes b_2)(a\otimes1-1\otimes a)^k$. But
\begin{align*}(b_1\otimes b_2)(a\otimes 1-1\otimes a)^k=&(1\otimes b_2)(b_1a\otimes 1-1\otimes b_1a)(a\otimes1-1\otimes a)^{k-1}\\
&-(1\otimes ab_2)(b_1 \otimes 1-1\otimes b_1)(a\ot 1 - 1\ot a)^{k-1},\end{align*}
so we are done. \end{proof}
We let $R=\mathbb{C}[x_1,\ldots,x_n]^{\otimes m}=\mathbb{C}[x_1^1,\ldots,x_n^1,x_1^2,\ldots,x_n^2,\ldots,x_1^m,\ldots,x_n^m]$.  Let $J_j$, for $1\leq j \leq m-1$ be the ideal in $R$ generated by $X^{j}_i:=x^{j}_i-x^{j+1}_i$. Let $$J=\displaystyle \sum_{j=1}^{m-2}J^3_j+J_{m-1}^2.$$
\begin{cor}
$K'_0=JY,$ so $Y/K'_0=Y/JY$.
\end{cor}
\begin{proof}This follows immediately from Lemma \ref{cor1} and Lemma \ref{cor2}.\end{proof}
Now, we can finish the proof of the theorem. Namely, $Y$ is a free module over $R$, so $h_{Y/JY}=h_{R/J}\cdot h_{X}$, where $h_X$ is the Hilbert series of the generators over $R$ of $Y$.  For $n=2,3$, we have $h_X=1$, while for $n\geq 4$, we have:
\begin{equation}\label{hX}h_X= (1+\sigma_2)^{m-2}\cdot\sum_{j+k\leq 2\lfloor\frac{n-2}{2}\rfloor} \sigma_{2j}\cdot\sigma_{2k}.\end{equation}
where $\sigma_l=\displaystyle \sum_{i_1<\cdots<i_l} t_{i_1}\cdots t_{i_l}$ are elementary symmetric functions. Now, from the description of $J$, we compute

$$h_{R/J}=\frac{(1+\displaystyle \sum t_i+\displaystyle\sum_{i\leq j}t_it_j)^{m-2}(1+\displaystyle\sum t_i)}{(1-t_1)\cdots(1-t_n)}.$$ Thus $h_{Y/JY}\cdot(1-t_1)\cdots(1-t_n) $ is a polynomial of degree less than or equal to $2m-3$, for $n=2,3$ and $2m-3 +2(m-2) + 2\lfloor\frac{n-2}{2}\rfloor$ for $n\geq 4$, proving Theorem \ref{2m-3}. 

\begin{rem}
For n=2, it follows from the proof of the theorem that the image, $f_m(v_m)$, of 
$$v_m = (x_1-x_2)(y_1-y_2)\cdots(x_{m-2}-x_{m-1})(y_{m-2}-y_{m-1})(x_{m-1}-x_m)$$
generates a copy of $(m-1,m-2),$ if it is non-zero (here we have used the notation $x_k:=x_1^k, y_k:=x_2^k$).  We conjecture that $f_m(v_m)$ is indeed non-zero, and that $(m-1,m-2)$ occurs with multiplicity one in $B_m(A_2)$.  Furthermore we conjecture that for $(p,q)$ occurring in $B_m(A_2)$, one has $p\leq m-1$.  For $m\leq 7$, these conjectures are confirmed in Theorem \ref{JH}.
\end{rem}

\section{The complete structure of $B_3(A_n)$}\label{PavelConjPf}
In this section, we prove Theorem \ref{PavelConj}, which was first conjectured by P. Etingof.  Firstly, we show that only representations of the form $(2,1^{2i-1},0^{n-2i})$ can appear in the Jordan-H\"older series of $B_3(A_n)$; this is accomplished by a strengthening of Theorem \ref{2m-3} in the case $m=3$.  Secondly, we use the representation theory of $\mathfrak{gl}_n$ to bound the multiplicity of each $(2,1^{2i-1},0^{n-2i})$ in $B_3(A_n)$ to at most one.  Thirdly, we exhibit a non-zero vector in each $(2,1^{2i-1},0^{n-2i})$ by representing the generators of $A_n$ in a certain quotient algebra.  Finally we show that the sum appearing in the Theorem is direct.  Steps 2-4 were explained to us by P. Etingof.

\subsection{Step one.}
 \begin{lem} \label{simplemult}
 Let $m \geq 3$. Then no $\mathcal{F}_{(1^k,0^{n-k})}$ occurs in $B_m$. 
 \end{lem}
\begin{proof}
\textbf{Case 1: $k=n$.}\\
As a representation of $S_n$, the polylinear part (i.e., the part of degree 1 in each variable) of $(1,\ldots,1)$ is isomorphic to the sign representation. On the other hand, the polylinear part of $A_n$ (i.e., the span of monomials of the form $x_{\sigma(1)}\cdots x_{\sigma(n)}$ for $\sigma \in S_n$) is clearly a copy of the regular representation, which contains the sign representation exactly once, and thus the total multiplicity of $(1,\ldots,1)$ in all $B_{m}(A_n)$ is equal to 1. In \cite{FS}, it is shown that when $n$ is odd, $(1,\ldots,1) \subset \bar{B}_1(A_n)$ and when $n$ is even, $(1,\ldots,1) \subset {B}_2(A_n)$, so $(1,\ldots,1)$ cannot occur in $B_m$ for $m\geq3$.\\
\textbf{Case 2: $k<n$.}\\
Suppose that $(1^k,0^{n-k})$ occurs in $B_m(A_n)$. If $k<n$, then this means $B_m(A_{k})$ contains a copy of $(1^k)$, which contradicts Case 1.
\end{proof}

\begin{lem}\label{B3identity} We have:
\begin{equation*}[x[y,z],[w,v]]=[x,[w[y,z],v]] - [y,[w[x,z],v]]+[z,[w[x,y],v]]\mod L_4\end{equation*}
\end{lem}
\begin{proof}
Let $G$ denote the symmetric group on the set $\{x,y,z,w,v\}$, and let $\psi$ denote the RHS minus the LHS of the identity.  Applying the Jacobi identity to the term $[x[y,z],[w,v]]$, we have
\begin{align*}
\psi =[x,[w[y,z],v]] - [y,[w[x,z],v]] + [z,[w[x,y],v]] - [w,[x[y,z],v]] + [v,[x[y,z],w]].
\end{align*}
Recall the isomorphism $\xi:\Omega^{ev,\geq 2}_{ex}\overset\sim\to B_2$ of Theorem \ref{isom}; for any $\alpha,\beta,\gamma,\delta\in \bar{B}_1$, we have $[\alpha[\beta,\gamma],\delta]=4\xi(d\alpha\wedge d\beta \wedge d\gamma \wedge d\delta) \mod L_3$.  Thus we may re-express $\psi$:
\begin{align*}
\psi =& 4\big( [x,\xi(dw\wedge dy \wedge dz \wedge dv)] - [y,\xi(dw\wedge dx \wedge dz\wedge dv)] + [z,\xi(dw\wedge dx \wedge dy \wedge dv)]\\ &- [w,\xi(dx\wedge dy \wedge dz \wedge dv)] + [v,\xi(dx \wedge dy \wedge dz \wedge dw)]\big).
\end{align*}
We see immediately that $\psi$ is skew symmetric with respect to each of the permutations $(x,y)$, $(y,z)$, $(z,w)$ and $(w,v)$.  As these permutations generate $G$, it follows that if $\psi$ is non-zero, then $\CC\psi$ is isomorphic to the sign representation.  However, we have already seen in the proof of Lemma \ref{simplemult} that the sign representation does not occur in the polylinear part of $B_m(A_n)$ for any $m\geq 3$, and thus $\psi=0 \mod L_4$.
\end{proof}
\begin{cor} \label{B3case}$B_3(A_n)=\sum_{i}[x_i,B_2] + \sum_{i\leq j}[x_ix_j,B_2]$.
\end{cor}
\begin{proof}
Lemma \ref{B3identity}, combined with Corollary \ref{symm}, allows us to reduce the degree of any expression in the outer slot of degree three or greater, thus leaving only the quadratic terms, as desired.
\end{proof}

\begin{lem} \label{B3bound} For $\mathcal{F}_\lambda$ appearing in the Jordan-H\"older series of $B_3(A_n)$, we have 
$$|\lambda|\leq 3 + 2\lfloor\frac{n-2}{2}\rfloor=\left\{\begin{array}{ll}n, &\textrm{n odd}\\n+1, &\textrm{n even}\end{array}\right.$$
\end{lem}
\begin{proof}
By Corollary \ref{B3case}, the map $f_3$ in the proof of Theorem \ref{2m-3} is surjective when restricted to $Y:=(\Omega^0)\ot(\oplus_{j+k\leq \lfloor \frac{n-2}{2}\rfloor}\Omega^{2j}\ot\Omega^{2k})$.  Thus we may omit the factor $(1+\sigma_2)$ in equation (\ref{hX}), and we compute that $h_{Y/JY}\cdot(1-t_1)\cdots(1-t_n)$ is a polynomial of degree less than or equal to $3+2\lfloor\frac{n-2}{2}\rfloor$.
\end{proof}

\begin{cor} If $\mathcal{F}_\lambda$ appears in the Jordan-H\"older series of $B_3(A_n)$, then $\lambda=(2,1^{2i-1},0^{n-2i})$ for some $1\leq i \leq \lfloor\frac{n}{2}\rfloor$.
\end{cor}
\begin{proof}
We proceed by induction on $n$, the case $n=2$ having been proved in \cite{DKM}.  First, suppose for contradiction that some $\mathcal{F}_\lambda$ occurs in the Jordan-H\"older series for $B_3(A_n)$, with $\lambda_1\geq 3$.  Then we have $\lambda_n=0$ by Lemma \ref{B3bound}, which implies that $(\lambda_1,\ldots,\lambda_{n-1})$ occurs in $B_3(A_{n-1})$.  This contradicts the induction assumption.  Thus $\lambda_1\leq 2$.

Let us again suppose for contradiction that some $\mathcal{F}_\lambda$ occurs with \mbox{$\lambda_1=\lambda_2=2$}.  Then $\lambda_n=0$, and so $(2,2,\lambda_3,\ldots,\lambda_{n-1})$ occurs in $B_3(A_{n-1})$, which contradicts the induction assumption.

Furthermore, by Lemma \ref{simplemult}, no factors $(1^k,0^{n-k})$ may occur in $B_3$.  The only remaining possibilities are of the form $(2,1^k,0^{n-k-1})$,and it remains only to show that $k$ must be odd.  Indeed, if $k$ is even, say $\lambda=(2,1^k,0^{n-k-1})$, and $\mathcal{F}_\lambda$ occurs in $B_3(A_n)$, then $(2,1^k)$ occurs in $B_3(A_{k+1})$, which contradicts Lemma \ref{B3bound}.
\end{proof}

\subsection{Step two.}
\begin{lem}\label{multbound} The multiplicity of $(2,1^{n-1})$ in $B_3(A_n)$ is at most one.\end{lem}
\begin{proof}
Let $n=2k$.  By Lemma \ref{B3case}, we have a surjection:
$$f_3:\sum_{p=0}^{k-1} Y_p\to B_3,$$
where $Y_p= \Omega_0 \ot\Omega_0 \ot\Omega_{2p}$.  As in the proof of Theorem \ref{2m-3}, we let
$$R=\CC[x_1^1,\ldots, x_n^1, x_1^2,\ldots x_n^2, x_1^3,\ldots,x_n^3].$$
We can identify $Y_p$ with the free module over $R$ with generators 
$$X_p=\{dx_{\alpha_1}\wedge\cdots\wedge dx_{\alpha_{2p}}\},$$ and Hilbert series $h_{X_p}=\sigma_{2p}$.   We let $J_j$, for $j =1,2$, denote the ideal generated by $X^j_i:=x^j_i-x^{j+1}_i, i=1,\ldots n$, and let $J=J_1^3 + J_2^2$.  Then, Lemmas \ref{cor1} and \ref{cor2} imply that $JY_p$ is in the kernel of $f_3$.  As $J$ is $W_n$ invariant, we have a surjection of $W_n$ modules $f_3:Y_p/JY_p\to f_3(Y_p)\subset B_3$.  We compute:

$$h_{Y_p/JY_p}=h_{R/J}h_{X_p}=\frac{(1+\displaystyle \sum t_i+\displaystyle\sum_{i\leq j}t_it_j)(1+\displaystyle\sum t_i)\sigma_{2p}}{(1-t_1)\cdots(1-t_n)}.$$ Thus $h_{Y_p/JY_p}\cdot(1-t_1)\cdots(1-t_n) $ is a polynomial of degree less than or equal to $2p+3$, so that the maximal size for $\lambda$ which can appear in each $Y_p/JY_p$ is $2p+3$.  Thus only $Y_{k-1}$ can contribute to multiplicity of $(2,1^{n-1})$.

So we consider the image $f_3(Y_{k-1}) \subset B_3$.  Scalars in the outer slot are sent to zero, and we can view the inner two slots naturally as the degree $n$ subspace of $B_2\cong\Omega_{ex}^{ev,\geq 2}$. We can thus write $f_3(Y_{k-1})$ as a quotient of
$$M = (\oplus_{j\ge 1}S^jV)\otimes \Lambda^n V\otimes (\oplus_{m\ge 0} S^mV).$$
The multiplicity of the $\mathfrak{gl}_n$ module $(2,1^{n-1})$ in $M$ is equal to the multiplicity
of $(1,0^{n-1})$ in $(\Lambda^nV)^*\ot M = (\oplus_{j\ge 1}S^jV)\otimes (\oplus_{m\ge 0} S^mV)$, which is clearly one.
\end{proof}

\begin{cor}A cyclic generator of $(2,1^{n-1})$ in $B_3(A_n)$ is $$v_n=[x_1,[x_1,x_2]\cdots[x_{2k-1},x_{2k}]].$$\end{cor}
\begin{proof} Apply $f_3$ to the generator $x_1\ot x_1\wedge\cdots\wedge x_n$ of $(2,1^{n-1})$ in $M$.
\end{proof}

\begin{cor} The multiplicity of $(2,1^{2i-1},0^{n-2i})$ in $B_3(A_n)$ is at most one.\end{cor}
\begin{proof} If we assume to the contrary that $(2,1^{2i-1},0^{n-2i})$ appears with multiplicity greater than one, then it follows that $(2,1^{2i-1})$ occurs in $B_3(A_{2i})$ with multiplicity greater than one, contradicting Lemma \ref{multbound}.
\end{proof}

\subsection{Step three.}
Fix some $k\in \NN$, let $A$ be any algebra, let $E$ denote the exterior algebra in generators $\zeta_1,\ldots \zeta_k$, and let $B=A\ot E$.  We denote by $E^i,E_+,E_-,$ and $E_+^{\geq j}$, and $E_-^{\geq j}$ the $i$-th graded component, even part, odd part, and the even and odd parts of degree at least $j$, respectively.

\begin{lem} We have
\begin{align*}
 [B,B]=&[A,A]\otimes (E^0\oplus E_-) \oplus A\otimes E_+^{\ge 2}.\\
 [B,[B,B]]=&[A,[A,A]]\otimes (E^0\oplus E^1)\oplus A[A,A]\otimes E_+^{\ge
2}\oplus [A,A]\otimes E_-^{\ge 3}.\\
 [B,[B,[B,B]]]=&[A,[A,[A,A]]\otimes (E^0\oplus E^1)\oplus
([A,A[A,A]]+A[A,[A,A]])\otimes E^2\\ &\oplus [A,A[A,A]]\otimes E_-^{\ge 3}\oplus
A[A,A]\otimes E_+^{\ge 4}. \nonumber
\end{align*}
\end{lem}
\begin{proof} A direct computation using the skew commutativity of $E$.
\end{proof}

\begin{cor} We have
\begin{align*}
[B,[B,B]]/[B,[B,[B,B]]]=&([A,[A,A]]/[A,[A,[A,A]]])\otimes
(E^0\oplus E^1)\\&\oplus (A[A,A]/([A,A[A,A]]+A[A,[A,A]]))\otimes E^2\\&\oplus
([A,A]/[A,A[A,A]])\otimes E_-^{\ge 3}.
\end{align*}
\end{cor}

\begin{prop}  The generator $v_n=[x_1,[x_1x_2]\cdots[x_{2k-1},x_{2k}]]$ of $(2,1^{n-1})$ is nonzero in $B_3(A_n)$.
\end{prop}
\begin{proof}
Clearly it suffices to find some algebra $C$ and elements $x_1,\ldots x_{2k}$ such that the expression defining $v_n$ is not in $L_4(C)$.  We let $A$ be the free algebra in two generators $a,b,$ let $E$ be the exterior algebra in generators $\zeta_0,\ldots,\zeta_{2k}$,and let $B=A\otimes E$.
We set $x_i=\zeta_i$ for $i=2,\ldots,2k,$ and $x_{1}=a\zeta_{0}+b\zeta_{1}$. Then a direct
computation shows:
$$[x_1,[x_1x_2]\cdots[x_{2k-1}x_{2k}]]=2^{k+1}[ab]\otimes \zeta_0\wedge\cdots\wedge\zeta_{2k}.$$

 By the corollary, this is nonzero in $[B,[B,B]]/[B,[B,[B,B]]]$, as its component in $([A,A]/[A,A[A,A]])\otimes E_-^{\ge 3}$ is non-zero.  The proposition is proved.\end{proof}

\begin{cor} The Jordan-H\"older series of $B_3(A_n)$ is $\{(2,1^{2i-1},0^{n-2i})\}_{1\leq i\leq\lfloor\frac{n}{2}\rfloor}$, each appearing with multiplicity one. 
\end{cor}

\subsection{Step four.}
\begin{prop} Each $(2,1^{2i-1},0^{n-2i})$ is a submodule, so the sum in Theorem \ref{PavelConj} is direct. 
\end{prop}
\begin{proof}
Let $v_k=[x_1,[x_1x_2]\cdots[x_{2k-1}x_{2k}]]\in B_3(A_n)$, and let $X_k$ be the submodule generated by $v_k$. Clearly, $\partial_iv_k=0$ for all $i$, so the JH series of $X_k$ involves only terms of the form $(2,1^{2r-1},0^{n-2k})$, where $r\geq k$. On the other hand, we saw in the proof of Lemma \ref{multbound} that this representation cannot involve $\mathcal{F}_\lambda$ with more than $2k+1$ boxes.  Thus, we must have $X_k=(2,1^{2k-1},0^{n-2k})$ as desired.
\end{proof}

Corollary \ref{BConj} can now be derived by counting the dimension of the graded component for the decomposition in Theorem \ref{PavelConj}.

\subsection{A geometric description of the bracket of $\bar{B}_1$ and  $B_2$.}
The isomorphism in Theorem \ref{PavelConj} allows us the following geometric description of the bracket map,
\begin{align*}[-,-]:&(\bar{B}_1/\CC) \ot B_2\to B_3.\\&a\ot b \mapsto [a,b].\end{align*}  

To begin, we identify $\bar B_1/\CC\cong\Omega_{ex}^{odd}$, and $B_2\cong\Omega_{ex}^{ev,\ge
2}$, as in \cite{FS}.  Also, we identify $B_3$ with the direct sum of Theorem \ref{PavelConj}, by sending
$[x_1,[x_1,x_2]\cdots[x_{2k-1},x_{2k}]]$
to $\pi(dx_1\otimes dx_1\wedge\cdots\wedge dx_{2k}),$ where $\pi:V\ot\Lambda^{2k}V\to (2,1^{2k-1},0^{n-2k})$ is the standard projection.\footnote{Here, to fix normalizations unambiguously, we define the $\mathfrak{gl}(V)$-modules $(2,1^{p-1},0,\ldots,0)$ as the unique submodules of the corresponding type in $V\otimes \Lambda^p V.$}

\begin{defn}Let $\psi_s: \Lambda^{s+1} V\otimes \Lambda^q V \to (2,1^{s+q-1},0,\ldots,0)$ be the unique (up to scaling) morphism of $\mathfrak{gl}_n$-modules:
$$\psi_s(v_0\wedge\cdots\wedge v_s\otimes b)=\sum_{i=0}^s (-1)^i \pi(v_i\otimes
(v_0\wedge\cdots\hat v_i\cdots\wedge v_s\wedge b))$$
\end{defn}

\begin{prop}  For $a\in \Omega_{ex}^{2p+1}, b\in \Omega_{ex}^{ev,\ge 2},$ the bracket map is given by
the formula:
$$[a,b]=\psi_{2p}(a\otimes b).$$
\end{prop}
\begin{proof}
This follows immediately from Lemma 5.1 by induction on p.
\end{proof}
In particular, the proposition implies that the bracket map is induced from a fiberwise morphism of the corresponding vector bundles on $\mathbb{A}^n$.

\section{Decompositions}\label{decomps}
 Using the computational algebra system MAGMA \cite{BCP}, we were able to produce
 the bigraded Hilbert series of $B_{m}(A_{2})$ up to degree $12$, and the tri-graded Hilbert series for  $B_{m}(A_{3})$ up to degree $8$.  Combined with Theorem \ref{2m-3}, these results imply the following:
 
\begin{thm}\label{JH}  The Jordan-H\"{o}lder series of $B_{m}(A_{2})$ for $m=2,\ldots,7$ are
\begin{align*}
B_{2}(A_{2}) =& (1,1) \textrm{\cite{FS}}\\
B_{3}(A_{2}) =& (2,1) \textrm{\cite{DKM}}\\
B_{4}(A_{2}) =& (3,1)+(3,2) \textrm{\cite{DKM}}\\
B_{5}(A_{2}) =& (4,1)+(3,2)+(4,2)+(4,3)\\
B_{6}(A_{2}) =& (5,1)+(4,2)+(3,3)+2(5,2)+2(4,3)+(5,3)+(5,4)\\
B_{7}(A_{2}) =&
(6,1)+2(5,2)+2(4,3)+2(6,2)+3(5,3)+2(4,4)+2(6,3)+2(5,4)\\ &+(6,4)+(6,5)
\end{align*}
The Jordan-H\"{o}lder series of $B_{m}(A_3)$ for $m=2,\ldots,5$ are
\begin{align*}
B_{2}(A_{3}) &=(1,1,0)\textrm{\cite{FS}}\\
B_{3}(A_{3}) &=(2,1,0)\textrm{\cite{DE}}\\
B_{4}(A_{3}) &=(3,1,0)+(2,1,1)+(3,2,0)+(2,2,1) \quad\textrm{\emph{(conj. in \cite{FS})}}\\
B_{5}(A_{3}) &=
  (4,1,0)+(3,2,0)+(3,1,1)+(2,2,1)+(4,2,0)+(4,1,1)\\&+
	3(3,2,1)+(2,2,2)+(4,3,0)+(3,3,1) \end{align*}
\end{thm}
\pagestyle{empty}


\begin{thebibliography}{99}
%\begin{singlespace}
%\bibitem[BR]{BR} A. Berenstein and V. Retakh. Lie algebras and Lie groups over noncommutative rings. arXiv:math/0701399.
\bibitem[BCP]{BCP}Wieb Bosma, John Cannon, and Catherine Playoust. The Magma algebra system. I. The user language. J. Symbolic Comput., 24(3-4):235-265, 1997 
\bibitem[DE]{DE} G. Dobrovolska and P. Etingof. An upper bound for the lower
  central series quotients of a free associative algebra. International
  Mathematics Research Notices, Vol. 2008, rnn039.
\bibitem[DKM]{DKM} G. Dobrovolska, J. Kim, X. Ma. On the lower central series
  of an associative algebra. arXiv:0709.1905.
%\bibitem[EKM]{EKM} P. Etingof, J. Kim, X. Ma. On universal Lie nilpotent
  %associative algebras. arXiv/0805.1909
%\bibitem[FF]{FF} B. Feigin, D. B. Fuks. Cohomologies of Lie Groups and Lie Algebras, in A.L. Onishchik, E.B. Vinberg (Eds.), \emph{Lie Groups and Lie Algebras II}. Springer, New York 2000.
\bibitem[FS]{FS} B. Feigin, B. Shoikhet. On $[A,A]/[A,[A,A]]$ and on a
  $W_{n}$-action on the consecutive commutators of free associative
  algebras. Math. Res. Lett. 14 (2007), no. 5, 781-795.
%\bibitem[F]{F} D. B. Fuks. \emph{Cohomology of Infinite-Dimensional Lie Algebras}. Consultants Bureau, New York (1986).
%\bibitem[FH]{FH} W. Fulton, J. Harris. \emph{Representation Theory: A First Course}. Springer Science+Business Media, Inc., New York, New York (2004).
%\bibitem[K]{K} M. Kapranov. Noncommutative geometry based on commutator 
%expansions. arXiv:math/9802041.
\bibitem[L]{SL} S. Lang. \emph{Differential and Riemannian Manifolds}. Springer-Verlag, New York, New York (1995).
\bibitem[R]{ANR} A. N. Rudakov. Irreducible representations of infinite-dimensional Lie algebras of Cartan type. Math. USSR Izv. Vol. 8, pgs. 836-866.
%\bibitem[S]{RS} R. P. Stanley. \emph{Enumerative Combinatorics}. Cambridge University Press, New York, New York (1997).
%\end{singlespace}
  \end{thebibliography}
\end{document}